\documentclass[12pt]{article}
\usepackage[utf8]{inputenc}
\usepackage{amssymb,amsmath,amsthm,amsfonts,fullpage,float,latexsym,bbm,microtype,cite}
\usepackage{fullpage} 
\usepackage{array}
\usepackage[colorlinks]{hyperref} 
\usepackage[parfill]{parskip} 
\usepackage{enumitem} 
\usepackage{graphicx}
\usepackage{xcolor}
\usepackage{mathtools}
\usepackage{dsfont}
\usepackage{comment}

\newtheorem{theorem}{Theorem}[section]

\newtheorem{lemma}[theorem]{Lemma}
\newtheorem{corollary}[theorem]{Corollary}

\newtheorem{conj}[theorem]{Conjecture}

\makeatletter 
\newtheorem*{rep@theorem}{\rep@title}\newcommand{\newreptheorem}[2]{%
\newenvironment{rep#1}[1]{%
\def\rep@title{\bf #2 \ref{##1}}%
\begin{rep@theorem}}%
{\end{rep@theorem}}}
\makeatother
\newreptheorem{theorem}{Theorem}

\newtheorem*{rep@corollary}{\rep@title}\newcommand{\newrepcorollary}[2]{%
\newenvironment{rep#1}[1]{%
\def\rep@title{\bf #2 \ref{##1}}%
\begin{rep@corollary}}%
{\end{rep@corollary}}}
\makeatother
\newreptheorem{corollary}{Corollary}
\newtheorem{definition}[theorem]{Definition}

\theoremstyle{remark}
\newtheorem*{remark}{Remark}

\newcommand{\prob}{\mathbb{P}}

\newcommand{\naturals}{\mathbb{N}}

\newcommand{\eps}{\varepsilon}

\title{Asymmetric Ramsey properties of randomly perturbed graphs}

\author{Emily Heath\thanks{Department of Mathematics, Iowa State University, Ames IA.\newline \indent\,\,\,\,Email: \texttt{\{eheath, dam1\}@iastate.edu}} \and 
Daniel McGinnis\footnotemark[1]
}

\begin{document}
\maketitle

\begin{abstract}
In this note, we investigate for various pairs of graphs $(H,G)$  the question of how many random edges must be added to a dense graph to guarantee that any red-blue coloring of the edges contains a red copy of $H$ or a blue copy of $G$. We determine this perturbed Ramsey threshold for many new pairs of graphs and  various ranges of densities, obtaining several generalizations of results obtained by Das and Treglown. In particular, we resolve the remaining cases toward determining the perturbed Ramsey threshold for pairs $(K_t,K_s)$ where $t\geq s\geq 5$. 
\end{abstract}

\section{Introduction}

Given graphs $H_1,\ldots, H_r$, we say that a graph $G$ is \emph{$(H_1,\ldots, H_r)$-Ramsey} if any $r$-coloring of the edges of $G$ contains a monochromatic copy of $H_i$ in color $i$, for some $i\in [r]$. In the case $H_1=\cdots=H_r$, we refer to such a graph as being $(H,r)$-Ramsey. The \emph{Ramsey number} $R(H_1,\ldots,H_r)$ is the smallest number of vertices needed so that $K_n$ is $(H_1,\ldots,H_r)$-Ramsey. Since Ramsey~\cite{R} showed that this number exists in 1930, the question of improving the bounds on $R(H_1,\ldots,H_r)$ has received significant attention. 
R\"odl and Ruci\'nski~\cite{RR1,RR2,RR3} initiated the study of Ramsey properties of the \emph{random graph} $G(n, p)$ which has vertex set $[n]$ and and where each edge is present with probability $p$, independently of all other choices. They determined the threshold for the $(H, r)$-Ramsey property in $G(n, p)$ for all fixed graphs $H$ and $r\geq 2$, showing that it depends on the so-called 2-density of the graph. 

For a graph $H$, let $d_2(H)= 0$ if $e(H) = 0$,  $d_2(H) = 1/2$ if $H$ is an edge, and $d_2(H)=\frac{e(H)-1}{v(H)-2}$ otherwise. The \emph{2-density} of $H$ is given by $m_2(H) = \max_{H'\subseteq H} d_2(H')$. A graph $H$ is said to be \emph{strictly 2-balanced} if $m_2(H') < m_2(H)$ for all $H'\subsetneq H$. 

\begin{theorem}[\hspace{1sp}\cite{RR1,RR2,RR3}]\label{thm:RamseyThreshold}
Let $r\geq 2$ be a positive integer and $H$ be a graph that is not a
forest consisting of stars and paths of length 3. There are positive constants $c$ and $C$ such that 
\[\lim_{n\rightarrow\infty} \prob[G(n,p) \text{ is } (H,r)\text{-Ramsey}]= \begin{cases}
0, & p < cn^{-1/m_2(H)}, \\
1, & p > Cn^{-1/m_2(H)}.
\end{cases}\]
\end{theorem}
Nenadov and Steger~\cite{NS} later gave a short proof of this result using the hypergraph
container method~\cite{BMS,ST}.

More generally, Kohayakawa and Kreuter~\cite{KK} conjectured that for any graphs $H_1,\ldots,H_r$, the threshold for $G(n,p)$ to be $(H_1,\ldots,H_r)$-Ramsey is governed by the \emph{asymmetric density} of the two densest graphs. The asymmetric density of graphs $H_1$ and $H_2$ with $m_2(H_1)\geq m_2(H_2)$ is 
\begin{equation}\label{assym}
m_2(H_1, H_2) = \max\left\{\frac{e(H_1')}{v(H_1')-2+1/m_2(H_2)} : H_1'\subseteq H_1 \text{ and } e(H_1')\geq 1\right\}.
\end{equation}
We say that $H_1$ is \emph{strictly balanced with respect to $m_2(\cdot, H_2)$} if no $H'_1\subsetneq H_1$ with at least one edge maximizes (\ref{assym}).

\begin{conj}[\hspace{1sp}\cite{KK}]\label{conj:Gnpthreshold}
Let $r \geq 2$ and suppose that $H_1,\ldots, H_r$ are non-empty graphs such that $m_2(H_1) \geq m_2(H_2) \geq \cdots\geq m_2(H_r)$ and $m_2(H_2) > 1$. Then there exist
constants $c, C > 0$ such that
\[\lim_{n\rightarrow\infty} P[G(n,p) \text{ is } (H_1, \ldots, H_r)\text{-Ramsey}] = \begin{cases}
0, & p < cn^{-1/m_2(H_1,H_2)}, \\
1, & p > Cn^{-1/m_2(H_1,H_2)}.
\end{cases}\]
\end{conj}
If true, Conjecture~\ref{conj:Gnpthreshold} would generalize Theorem~\ref{thm:RamseyThreshold}, since $m_2(H_1)\geq m_2(H_1,H_2)\geq m_2(H_2)$ with equality if and only if $m_2(H_1)=m_2(H_2)$. As evidence towards this conjecture, Kohayakawa and Kreuter~\cite{KK} proved the result in the case $r=2$ where both graphs are cycles. The conjecture was later proved for the case where $H_1$ and $H_2$ are cliques in~\cite{MSSS} and the case where $H_1$ is a clique and $H_2$ is a cycle in~\cite{LMMS}. 

Following progress in~\cite{BMS,GNPSST,HST,KK,MSSS}, the full 1-statement of the conjecture was proved by Mousset, Nenadov and Samotij~\cite{MNS} using the container method. The 0-statement was proved by Hyde~\cite{Hyde} for almost all pairs of regular graphs, and by Kuperwasser and Samotij~\cite{KS} when $m_2(H_1)=m_2(H_2)$. Very recently, Bowtell, Hancock, and Hyde~\cite{BHH} and Kuperwasser, Samotij, and Wigderson~\cite{KSW} independently reduced the 0-statement of the conjecture to a deterministic coloring problem. In~\cite{BHH}, the authors resolved the conjecture in many cases, including when $r\geq3$, when $H_2$ is strictly 2-balanced and not bipartite, and when $m_2(H_1)=m_2(H_2)$, in addition to proving that an analogous conjecture holds for almost all pairs of uniform hypergraphs. On the other hand, in~\cite{KSW}, in addition to proving the conjecture for many cases, the authors extended their results to families of graphs. 

While the results above address the Ramsey properties of typical graphs of a given density, we are interested in this note in the related question of how far a dense graph can be from satisfying a Ramsey property. Bohman, Frieze and Martin~\cite{BFM} introduced the model of randomly perturbed graphs to study questions of this type, beginning with a dense graph and asking how many random edges must be added to ensure that the resulting graph will satisfy a given property with high probability. This randomly perturbed graph model has been extensively studied \cite{ADHL2,ADHL,BFP,BTW,BDF,BFKM,BHKMPP,BMPP,DMT,CHKMM,DRRS,JK,KL,KKSspanning,KST,NT} and generalized to the settings of directed graphs, hypergraphs, and sets of integers~\cite{AP,BHKM,DKM,HZ,KKSdigraphs,MM}.

Krivelevich, Sudakov, and Tetali~\cite{KST} initiated the study of Ramsey properties of randomly perturbed graphs, asking how many random edges must be added to any dense graph to ensure with high probability that the resulting graph is $(H_1, H_2)$-Ramsey. They answered this question in the case when $H_1 = K_t$ and $H_2 = K_3$ for $t\geq 3$. Their results were generalized to larger cliques by Das and Treglown~\cite{DT} and Powierski~\cite{P}, resolving the $(K_t, K_s)$-Ramsey problem in the range of densities $0<d\leq \frac{1}{2}$ for all cases except $s=4$ and $t\geq 5$, which appears to be quite difficult. In fact, in \cite{DT}, the authors addressed the following more refined question: given any fixed density $0 < d < 1$, how many random edges must be added to any graph $G$ of density at least $d$ to ensure that with high probability the resulting graph is $(H_1, H_2)$-Ramsey? 

In addition to settling the question for most cliques, Das and Treglown~\cite{DT} determined the perturbed Ramsey threshold for all pairs of cycles and all densities. In each case, their result showed that significantly fewer random edges are needed for the perturbed $(C_k, C_{\ell})$-Ramsey question compared to the result in the random graph setting. They also determined the perturbed Ramsey thresholds for odd cycles versus cliques. However, the  question of Krivelevich, Sudakov, and Tetali~\cite{KST} remains open for all other pairs of graphs. 

In this note, we determine the perturbed Ramsey threshold for several additional classes of graphs and densities. To state our results, we will need the following definition. 

\begin{definition}
Given a density $0<d<1$, a number of colors $r\in\naturals$, and a sequence of graphs $(H_1,\ldots,H_r)$, the \emph{perturbed Ramsey threshold probability} $p(n;H_1,\ldots,H_r,d)$ satisfies the following:
\begin{enumerate}
    \item If $p=p(n)=\omega(p(n;H_1,\ldots,H_r,d))$, then for any sequence $(G_n)_{n\in\naturals}$ of $n$-vertex graphs with density at least $d$, the graph $G_n\cup G(n,p)$ is $(H_1,\ldots,H_r)$-Ramsey with high probability. 
    \item If $p=p(n)=o(p(n;H_1,\ldots,H_r,d))$, then for some sequence $(G_n)_{n\in\naturals}$ of $n$-vertex graphs with density at least $d$, the graph $G_n\cup G(n,p)$ is with high probability not $(H_1,\ldots,H_r)$-Ramsey.
\end{enumerate}
If it is the case that every sufficiently large graph of density at least $d$ is $(H_1,\ldots,H_r)$-Ramsey, then we define $p(n;H_1,\ldots,H_r,d)=0$. In the symmetric case where $H_1=\cdots=H_r=H$, we denote the threshold by $p(n;r,H,d)$.
\end{definition}

Our main results are Theorem \ref{thm:GeneralkPartition} and \ref{thm:main2} which determine the perturbed Ramsey threshold for pairs of graphs satisfying several technical conditions. We refer the statement of these theorems to Section \ref{sec:main}, and instead now state several applications to interesting families of graphs below. The proofs of these corollaries are presented in Section \ref{sec:applications}.

In \cite{DT}, the following perturbed Ramsey threshold result was shown for pairs of complete graphs.

\begin{theorem}[\hspace{1sp}\cite{DT}]
    Let $s,t,k$ be integers with $k\geq 2$ and $t\geq s\geq 2k+1$, and let $\frac{k-2}{k-1}\leq d\leq \frac{k-1}{k}$. If
    \begin{enumerate}
        \item[(i)] $k=2$, or
        \item[(ii)] $s\equiv 1$ (mod $k$),
    \end{enumerate}
    then $p(n;K_t,K_s,d) = n^{-1/m_2(K_t,K_{\lceil s/k \rceil})}$. Otherwise if
    \begin{enumerate}
        \item[(iii)] $k\geq 3$ and $s\not \equiv 1$ (mod $k$),
    \end{enumerate}
    then $p(n;K_t,K_s,d) = n^{-(1-o(1))/m_2(K_t,K_{\lceil s/k \rceil})}$.
\end{theorem}
In particular, an exact threshold function when $k\geq 3$ and $s\not \equiv 1$ (mod $k$) was not proven. We determine the exact threshold function even in these cases as a corollary of Theorem \ref{thm:GeneralkPartition}.

\begin{repcorollary}{cor:KtKs}
 Let $s,t,k$ be integers with $k\geq 2$ and $t\geq s\geq 2k+1$, and let $\frac{k-2}{k-1}\leq d\leq \frac{k-1}{k}$. Then we have that
    \[
    p(n;K_t,K_s,d) = n^{-1/m_2(K_t,K_{\lceil s/k \rceil})}.
    \]    
\end{repcorollary}

In addition, Theorem~\ref{thm:GeneralkPartition} yields a similar result for another pair of graphs: a clique $K_t$ and a graph $K_s'$ obtained by deleting a maximum matching from the complete graph $K_s$.

\begin{repcorollary}{cor:KsMinusMatching}
    Let $K_s'$ be the graph obtained by deleting a maximum matching from the complete graph $K_s$. For an integer $2\leq k < \frac{s}{3}$ and in the range of densities $\frac{k-2}{k-1}\leq d\leq \frac{k-1}{k}$, we have that for $t\geq s\geq 7$, 
    \[
    p(n;K_t,K_s',d) = n^{-1/m_2(K_t,K_{\lceil s/k \rceil}')}.
    \]
\end{repcorollary}

In \cite{DT}, Das and Treglown also determined the perturbed Ramsey threshold for pairs consisting of a complete graph and an odd cycle. Specifically, they prove the following.

\begin{theorem}[\hspace{1sp}\cite{DT}]\label{thm:KtCl}
    For any clique size $t\geq 4$, odd cycle length $\ell \geq 5$, and $0<d\leq 1/2$, we have that
    \[
    p(n;K_t,C_\ell,d) = n^{-2/(t-1)}.
    \]
\end{theorem}

As an application of Theorem \ref{thm:main2}, we generalize Theorem \ref{thm:KtCl} to pairs $(K_t,G)$ where $t\geq 5$ and $G$ is, for example, an odd cycle, an odd cycle with a chord, an even cycle with a chord
connecting two vertices an even distance away along the cycle, or a star on at least 3 vertices with an additional vertex that is adjacent to each vertex of the star.

\begin{repcorollary}{cor:oddcycleGen}
    Given a tree $T$ with at least 4 vertices, let $G$ be a graph obtained by adding a vertex $v$ to $T$ and edges from $v$ to $V(T)$ such that there is no 2-coloring of $T$ where the vertices in $N_G(v)$ are all the same color, but there is some vertex $u \in N_G(v)$ such that there is a 2-coloring of $T$ where the vertices of $N_G(v)\setminus \{u\}$ are colored with the same color.
    
    Then for $t\geq 5$ and $0<d\leq 1/2$,
    \[
    p(n;K_t,G,d) = n^{-2/(t-1)}.
    \]
\end{repcorollary}

Finally, we obtain the following result as another application of Theorem \ref{thm:main2}. This result determines the perturbed Ramsey threshold for some pairs of 3-chromatic graphs. In particular, unlike in the previous results, both graphs in the pair have significantly fewer edges than a complete graph.

\begin{repcorollary}{cor:last}
    Let $K$ be the graph obtained by taking a star on $t\geq 4$ vertices and adding a new vertex that is adjacent to each vertex of the star. Let $G$ be as in Corollary \ref{cor:oddcycleGen}. Then for $0<d\leq 1/2$,
    \[
    p(n;K,G,d) = n^{-(t+1)/(2t-1)}.
    \]
\end{repcorollary}

\section{Preliminaries}\label{sec:prelim}

In this section, we will introduce the tools used in our proofs. In addition to our main tool, the famous Szemer\'edi Regularity Lemma~\cite{KSReg}, which allows us to find useful substructures in the dense underlying graph $G$, we will also rely on other results to analyze the random edges in $G(n,p)$. 

\begin{definition}
Given $\eps >0$, a graph $G$ and two disjoint vertex sets $A,B\subset V(G)$, the pair $(A,B)$ is \textup{$\eps$-regular} if for every $X\subset A$, $B\subset Y$ with $|X| > \eps|A|$ and $|Y| > \eps |B|$, we have $|d(X,Y) - d(A,B)| < \eps$.
\end{definition}

We will make use of the following well-known properties of $\eps$-regular pairs.

\begin{lemma}\label{lem:LargeNeighborhoods}
Let $(A,B)$ be an $\eps$-regular pair in a graph $G$ and set $d=d(A,B)$.
\begin{enumerate}
    \item If $\ell \geq 1$ and $(d-\eps)^{\ell-1} > \eps$, then
    \[
    \left|\left\{(x_1,x_2,\dots,x_{\ell})\in A^\ell : |\cap_i N(x_i)\cap B|\leq (d-\eps)^\ell|B|\right\}\right|\leq \ell\eps|A|^\ell.
    \]
    \item If $\alpha > \eps$, and $A'\subset A$ and $B'\subset B$ satisfy $|A'|\geq \alpha|A|$ and $|B'|\geq \alpha|B|$, then $(A',B')$ is an $\eps'$-regular pair of density $d'$, where $\eps'=\textrm{max}\{\eps/\alpha,2\eps\}$ and $|d'-d|<\eps$.
\end{enumerate}
\end{lemma}



Combining the Regularity Lemma with Tur\'an's Theorem~\cite{T} yields the following corollary, which we will apply in our proofs to build monochromatic copies of graphs. 

\begin{theorem}\label{thm:EpsilonRegularSets}
For every $\eps,\delta>0$ with $\delta\geq 3\eps$, there is some $\eta=\eta(\eps,\delta)>0$ and $n_0=n_0(\eps,\delta)\in \mathbb{N}$ such that the following holds for all $n\geq n_0$ and $k\geq 2$. If $G$ is an $n$-vertex graph with density at least $1-\frac{1}{k-1}+\delta$, then there are pairwise disjoint vertex sets $V_1,\dots,V_k\subset V(G)$ such that $|V_1|=\cdots=|V_k|\geq \eta N$, and for each $1\leq i<j\leq k$ the edges between $V_i$ and $V_j$ form an $\eps$-regular pair of density at least $\frac{\delta}{2}$.
\end{theorem}

In addition to the results above for finding structure in dense graphs, we will need several tools for finding structure in random graphs. The first is the following  standard application of the Janson inequality \cite[Theorem 2.14]{JLR} proven in~\cite{DT}.

\begin{theorem}\label{thm:CopiesofH}
Let $H$ be a graph with $v\geq 2$ vertices and $e\geq 1$ edges. Let $[n]$ be the vertex set of $G(n,p)$, and for some $\xi > 0$, let $\mathcal{H}$ be a collection of $\xi n^v$ possible copies of $H$ with vertex set in $[n]$. The probability that $G(n,p)$ does contain a copy of $H$ from $\mathcal{H}$ is at most $\textrm{exp}(-\xi \mu_1/(2^{v+1}v!))$, where $\mu_1=\mu_1(H)=\textrm{min}\{n^{v(F)}p^{e(F)} \mid F\subseteq H, e(F)\geq 1\}$.
\end{theorem}

We will also employ results on Ramsey properties of random graphs. While the 1-statement of Conjecture~\ref{conj:Gnpthreshold} is now known for all graphs~\cite{MNS}, we need the following stronger Ramsey properties defined in~\cite{DT} which are satisfied by many pairs of graphs $(H_1,H_2)$ above the threshold of Conjecture~\ref{conj:Gnpthreshold}.

\begin{definition}\label{def:RobustandGlobalRamsey}
Let $H_1$ and $H_2$ be two fixed graphs, and let $G$ be an $n$-vertex graph on the vertex set $[n]$.

Given families $\mathcal{F}_i\subset \binom{[n]}{v(H_i)}$, for $i\in [2]$, of \textup{forbidden} subsets of $v(H_i)$ vertices, we say $G$ is \textup{robustly $(H_1,H_2)$-Ramsey with respect to $(\mathcal{F}_1,\mathcal{F}_2)$} if for every 2-coloring of $G$, there is a red copy of $H_1$ or a blue copy of $H_2$ whose vertex set is not forbidden (i.e. the vertex set of the monochromatic copy $H_i$ is not a member of $\mathcal{F}_i$).

Given $\mu >0$, we say that $G$ is \textup{$\mu$-globally $(H_1,H_2)$-Ramsey} if for every 2-coloring of $G$ and every subset $U\subseteq [n]$ of size at least $\mu n$, $G[U]$ contains a red copy of $H_1$ or a blue copy of $H_2$.
\end{definition}

By modifying the proofs of Hancock, Staden and Treglown\cite{HST} and of Gugelmann, Nenadov, Person, \v{S}kori\'c, Steger and Thomas~\cite{GNPSST} proved with hypergraph containers, Das and Treglown \cite{DT} showed that the random graph $G(n,p)$ is not just $(H_1,H_2)$-Ramsey but in fact robustly and globally Ramsey for many graphs $(H_1,H_2)$ beyond the threshold of Conjecture~\ref{conj:Gnpthreshold}.

\begin{theorem}[\hspace{1sp}\cite{DT}]\label{thm:StrongerRamseyPoperties}
Let $H_1$ and $H_2$ be graphs such that $m_2(H_1)\geq m_2(H_2) \geq 1$ and 
\begin{itemize}
    \item $m_2(H_1)=m_2(H_2)$, or
    \item $H_1$ is strictly balanced with respect to $m_2(\cdot,H_2)$.
\end{itemize}
The random graph $G(n,p)$ then has the following Ramsey properties.
\begin{enumerate}
    \item There are constants $\gamma=\gamma(H_1,H_2) > 0$ and $C_1=C_1(H_1,H_2)$ such that if $p \geq C_1n^{-1/m_2(H_1,H_2)}$ and if for $i\in [2]$, $\mathcal{F}_i\subset \binom{[n]}{v(H_i)}$ is a collection of at most $\gamma n^{v(H_i)}$ forbidden subsets, then $G(n,p)$ is with high probability robustly $(H_1,H_2)$-Ramsey with respect to $(\mathcal{F}_1,\mathcal{F}_2)$.
    \item For every $\mu>0$, there is a constant $C_2=C_2(H_1,H_2,\mu)$ such that if $p\geq C_2n^{-1/m_2(H_1,H_2)}$, then $G(n,p)$ is with high probability $\mu$-globally $(H_1,H_2)$-Ramsey.
\end{enumerate}
\end{theorem}

\section{Main Theorems}\label{sec:main}

In this section, we prove our two main results. Both of these theorems are stated somewhat generally, and we apply them to specific pairs of graphs in Section \ref{sec:applications}.

\begin{theorem}\label{thm:GeneralkPartition}
Let $k\geq 2$ and $K$, $G$, and $H$ be graphs satisfying the following properties. 
\begin{enumerate}
    \item For every partition $V_1,\dots,V_k$ of $V(G)$, $H$ is a subgraph of some $G[V_i]$.
    \item The graph obtained by taking $k$ disjoint copies of $H$ and adding all edges between the $k$ copies contains $G$ as a subgraph.
    \item $m_2(K) \geq m_2(G)\geq 1$, $m_2(K) \geq m_2(H)\geq 1$, and $K$ is strictly 2-balanced with respect to $m_2(\cdot,G)$ (respectively $m_2(\cdot,H)$) or $m_2(K) = m_2(G)$ (respectively $m_2(K) = m_2(H)$).
    \item There is a graph $K'$ obtained by deleting 1 vertex from $K$ such that $m_2(K,H)\geq m_2(K',G)$ and $m_2(K') \geq m_2(G)$, and $K'$ is strictly 2-balanced with respect to $m_2(\cdot,G)$ or $m_2(K') = m_2(G)$.
    \item The 0-statement of Conjecture \ref{conj:Gnpthreshold} holds for the pair of graphs $(K,H)$.
\end{enumerate}
Then for $\frac{k-2}{k-1} < d \leq \frac{k-1}{k}$,
\[
p\left(n; K, G, d\right)=n^{-1/m_2\left(K,H\right)}.
\]
\end{theorem}
\begin{proof}
First we  consider the lower bound. Let $p=o\left(n^{-1/m_2\left(K,H\right)}\right)$ and $\frac{k-2}{k-1} < d \leq \frac{k-1}{k}$. Let $\Gamma_n$ be the balanced complete $k$-partite graph on $n$ vertices, and color each edge of $\Gamma_n$ blue.  Since the pair of graphs $(K, H)$ satisfies the $0$-statement of Conjecture \ref{conj:Gnpthreshold}, there exists a 2-coloring of  $G(n,p)$ that does not have a red copy of $K$ or a blue copy of $H$. Color the remaining edges of $\Gamma_n \cup G(n,p)$ according to this coloring. Then $\Gamma_n\cup G(n,p)$ certainly contains no red $K$, and since any $k$-partition of $G$ contains $H$ in one of its partite sets, $\Gamma_n\cup G(n,p)$ also contains no blue $G$.

For the upper bound, let $p=\omega\left(n^{-1/m_2\left(K,H\right)}\right)$ and $\frac{k-2}{k-1} < d \leq \frac{k-1}{k}$. Let $t=|V(K)|$ and $h=|V(H)|$. By Theorem \ref{thm:EpsilonRegularSets} with $\delta=d-\frac{k-2}{k-1}$, and $\eps > 0$, for every $n$-vertex graph $\Gamma$ with $n$ sufficiently large, there exist $k$ disjoint vertex sets $V_1,\dots,V_k$ such that $|V_1| = \cdots = |V_k| = \eta n$ for some $\eta>0$, and for $1\leq i < j\leq k$ the pair $(V_i,V_j)$ is $\eps$-regular of density at least $\delta/2$. Let $\gamma=\gamma(K,H)$ be as in Theorem \ref{thm:StrongerRamseyPoperties}. 

We now prove by induction that with high probability, in any red-blue coloring of $\Gamma\cup G(n,p)$, either there is a red copy of $K$, or we may a find a blue copy of $H$ within each vertex set $V_i$ with all blue edges between the copies. This must necessarily contain a blue copy of $G$ as a subgraph, and hence will complete the proof. 

We begin by proving the base case, when $k=2$. Let $\Gamma\cup G(n,p)$ be 2-colored with red and blue. By Lemma \ref{lem:LargeNeighborhoods}, for sufficiently small $\eps$, there are at most $h\eps|V_1|^{h}$ $h$-sets in $V_1$ whose common neighborhood in $V_2$ in $\Gamma$ has size at most $(\delta/2-\eps)^{h}|V_2|$. Choose $\eps$ so that $h\eps < \gamma$. Let $\mathcal{F}_1=\emptyset$ and $\mathcal{F}_2$ be those $h$-sets of $V_1$ whose common neighborhood in $V_2$ in $\Gamma$ has size at most $(\delta/2-\eps)^{h}|V_2|$. We may assume that $G(n,p)[V_1]$ does not contain a red copy of $K$. Then since $G(n,p)[V_1] \sim G(\eta n,p)$, and since $m_2(K)\geq m_2(H)\geq 1$ and $K$ is strictly 2-balanced with respect to $m_2(\cdot, H)$ or $m_2(K)=m_2(H)$, Theorem~\ref{thm:StrongerRamseyPoperties} implies that there is a blue copy of $H$ in $G(n,p)[V_1]$ whose vertex set is not a member of $\mathcal{F}_2$.

Taking $\alpha=\left(\frac{\delta/2-\eps}{2}\right)^{h}$, if $U\subset V_2$ is the common neighborhood of the vertices in the blue copy of $H$ in $V_1$, then $|U|\geq\alpha2^{h}|V_2|=\alpha2^{h}\eta n$. By the pigeonhole principle, there exists a vertex set $U'\subset U$ of size at least $\alpha\eta n$ such that for each vertex $v\in V(H)$, the edges from $v$ to $U'$ are monochromatic. 

Since $m_2(K,H)\geq m_2(K',G)$, we have that $p=\omega\left(n^{-1/m_2(K',G)}\right)$. Now with $\mu=\alpha\eta$, Theorem \ref{thm:StrongerRamseyPoperties} implies that $G(n,p)$ is $\mu$-globally $(K',G)$-Ramsey with high probability. Therefore, with high probability, $G(n,p)[U']$ contains either a red copy of $K'$ or a blue copy of $G$. If there is a blue copy of $G$, then we are done, so assume there is a red copy of $K'$ in $G(n,p)[U']$. Then the edges between $V(H)$ and $U'$ must all be blue, as otherwise a red copy of $K$ is created. We also have that $G(n,p)$ is $(K,H)$-Ramsey with high probability, so we can assume that $G(n,p)[U']$ contains a blue copy of $H$. But then $\Gamma\cup G(n,p)$ contains the blue graph consisting of two disjoint copies of $H$ with all edges in between, which contains $G$ as a subgraph. Therefore, $\Gamma\cup G(n,p)$ is $(K,G)$-Ramsey with high probability.

Now assume $k\geq 3$. As before, let $\Gamma\cup G(n,p)$ be 2-colored with red and blue. For sufficiently small $\eps$ and for  $2\leq j\leq k$, there are at most $h\eps|V_1|^{h}$ $h$-sets in $V_1$ whose common neighborhood in $V_j$ in $\Gamma$ has size at most $(\delta/2-\eps)^{h}|V_j|$. Let $\eps$ be sufficiently small so that $(k-1)h\eps < \gamma$. Let $\mathcal{F}_1=\emptyset$ and $\mathcal{F}_2$ be those $h$-sets of $V_1$ whose common neighborhood in some $V_j$ in $\Gamma$ has size at most $(\delta/2-\eps)^{h}|V_j|$. Note that $|\mathcal{F}_2| \leq (k-1)h\eps|V_1|^{h} < \gamma |V_1|^{h}$. Therefore, similar to the case of $k=2$ above, we have that with high probability, in any red-blue coloring, there is a blue copy of $H$ whose vertex set is not a member of $\mathcal{F}_2$. For $2\leq j\leq k$, let $U_j\subset V_j$ be the common neighborhood of $V(H)$ in $V_j$ in $\Gamma$. As in the $k=2$ case, we can assume that there is a vertex set $U'_j\subset U_j$ for each $2\leq j\leq k$ of size $\alpha \eta n$ such that for each $v\in V(H)$, each edge from $v$ to $U'_j$ is blue. After decreasing $\eps$ if necessary so that $\alpha>\eps$, we have by Lemma \ref{lem:LargeNeighborhoods} that for $2\leq i< j\leq k$, the pair $(U'_i,U'_j)$ is $\eps'$-regular for $\eps' = \eps/\alpha$ and with density $\delta'>0$. By induction, we may assume with high probability that there is a blue copy of $H$ within each $U'_j$ and that all edges between these copies are blue. This together with the blue copy of $H$ from $V_1$ contains a blue copy of $G$ as a subgraph. This completes the proof.
\end{proof}

The following theorem provides a different set of conditions on pairs $(H_1,H_2)$ for which we can determine the associated perturbed Ramsey threshold. For the purpose of stating Theorem \ref{thm:main2} below, we denote $m(G) = \textrm{max}\{ e(H)/v(H) \mid H\subseteq G\}$ (the \textit{$m$-density of $G$}), and we say that $G$ is \textit{balanced} if $e(G)/v(G) = m(G)$.

\begin{theorem}\label{thm:main2}
Let $K$ and $G$ be graphs such that $m_2(K) \geq m_2(G) \geq 1$. Let $d(K) = \frac{e(K)}{v(K)}$ and assume that $K$ is balanced. Let $G$ be a $k$-chromatic graph with $k\geq 3$ such that $m(G) \leq d(K)$. Suppose the following conditions are satisfied.
\begin{enumerate}
    \item There exist graphs $K'$ and $G'$ obtained by deleting one vertex of $K$ and $G$, respectively, such that $d(K)\geq m_2(K',G')$, $m_2(K')\geq m_2(G') \geq 1$, and $K'$ is strictly 2-balanced with respect to $m_2(\cdot,G')$ or $m_2(K') = m_2(G')$.
    \item There exists a graph $K''$ obtained by deleting two vertices from $K$ such that $d(K)\geq m_2(K'',G)$, $m_2(K'') \geq m_2(G) \geq 1$ and $K''$ is strictly 2-balanced with respect to $m_2(\cdot,G)$ or $m_2(K'') = m_2(G)$.
    \item $G$ is a subgraph of the graph obtained by taking the complete balanced $(k-1)$-partite graph on $(k-1)|V(G)|$ vertices and adding $k-2$ additional edges, which are contained in distinct partite sets.
\end{enumerate}
Then for $\frac{k-3}{k-2}< d\leq \frac{k-2}{k-1}$, \[p(n; K,G,d)=n^{-1/d(K)}.\]
\end{theorem}
\begin{proof}
For the lower bound, let $p=o\left(n^{-1/d(K)}\right)$ and $\frac{k-3}{k-2}< d\leq \frac{k-2}{k-1}$. Define $\Gamma_n$ to be the complete balanced $(k-1)$-partite graph on $n$ vertices, and color the edges of $\Gamma_n$ blue. Color the remaining edges of $\Gamma_n \cup G(n,p)$ red. Since $n^{-1/d(K)}$ is the threshold function for $G(n,p)$ to contain $K$~\cite{Bollobas}, and since $G$ is $k$-partite, with high probability $\Gamma_n \cup G(n,p)$ contains no red $K$ or blue $G$ in this 2-coloring.

For the upper bound, let $p=\omega\left(n^{-1/d(K)}\right)$ and $\frac{k-3}{k-2}< d\leq \frac{k-2}{k-1}$.  Let $\Gamma$ be an $n$-vertex graph of density $d$, for sufficiently large $n$.  By Theorem \ref{thm:EpsilonRegularSets} with $\delta=d-\frac{k-3}{k-2}$ and $\eps>0$, there are vertex sets $V_1,\dots, V_{k-1}$ of size $\eta n$ such that $(V_i,V_j)$ is an $\eps$-regular pair of density at least $\delta/2$ for all $i<j$. Let $t=|V(K)|$ and  $m=|V(G)|$. Choose $\eps$ small enough so that $(\delta/2-\eps)^{t+m-1}>\eps$. By Lemma \ref{lem:LargeNeighborhoods} applied to the pairs of the form $(V_1,V_i)$, we can choose $\eps$ so that at least a $\frac{k-2}{k-1}$ proportion of the $(t+m)$-sets of $V_1$ have a common neighborhood in $\Gamma$ of size at least $(\delta/2-\eps)^{t+m}|V_i|$ in $V_i$.  Then at least a $\frac{1}{k-1}$ proportion of the $(t+m)$-sets of $V_1$ have a common neighborhood in $\Gamma$ of size at least $(\delta/2-\eps)^{t+m}|V_i|$ in $V_i$ for all $2\leq i\leq k-1$. 
    
Let $F$ be the disjoint union of $K$ and $G$. Let $\mathcal{H}$ be all the possible copies of $F$ on the $(t+m)$-sets in $V_1$ described above with large common neighborhoods in $\Gamma$ in each of $V_2,\ldots, V_{k-1}$. Then there is some $\xi > 0$ depending only on $t,m$, and $d$ such that $|\mathcal{H}|\geq \xi n^{t+m}$. Since $p=\omega\left(n^{-1/d(K)}\right)$, it follows from the fact that $K$ is balanced and $m(G)\leq d(K)$ that $\mu_1(F)=\omega(1)$, where $\mu_1$ is as in Theorem \ref{thm:CopiesofH}. By Theorem \ref{thm:CopiesofH}, the probability that $G(n,p)[V_1]$ does not contain a copy of $F$ with vertex set in $\mathcal{H}$ is at most exp$(-\xi\omega(1)/(2^{t+m+1}(t+m)!))=o(1)$. Thus, with high probability, $G(n,p)[V_1]$ contains a copy $F_1$ of $F$ whose common neighborhood, say $U_i$, in $V_i$ has size at least $(\delta/2-\eps)^{t+m}|V_i|$ for all $2\leq i\leq k-1$. Now take $\alpha$ so that $\alpha2^{t+m}=(\delta/2-\eps)^{t+m}$, so in particular, $|U_i|\geq \alpha2^{t+m}|V_i|$ for all $2\leq i\leq k-1$. Take an arbitrary 2-coloring of the edges of $\Gamma\cup G(n,p)$. By the pigeonhole principle, for each $2\leq i\leq k-1$ we can find a subset $U_i'\subset U_i$ of size at least $\alpha|V_i|$ such that for every vertex $v\in V(F_1)$, each edge $vu$ with $u\in U_i'$ has the same color. 

Since $d(K)\geq m_2(K',G')$ and $d(K)\geq m_2(K'',G)$, by Theorem \ref{thm:StrongerRamseyPoperties}, $G(n,p)$ is $\alpha\eta$-globally $(K',G')$-Ramsey and $\alpha\eta$-globally $(K'',G)$-Ramsey with high probability. Suppose there are vertices $u,v\in V(F_1)$ such that all edges from $u$ to $U_2'$ are red and all edges from $v$ to $U_2'$ are blue. Since $G(n,p)$ is $\alpha\eta$-globally $(K',G')$-Ramsey, $U_2'$ contains either a red copy of $K'$ or a blue copy of $G'$. However, if $U_2'$ contains a red copy of $K'$, then this creates a red copy of $K$ with $u$, and if $U_2'$ has a blue copy of $G'$, then this makes a blue copy of $G$ with $v$. Hence, we may assume that the edges from $V(F_1)$ to $U_2'$ are either all red or all blue. First assume that all of the edges from $V(F_1)$ to $U_2'$ are red. Since $G(n,p)$ is $\alpha\eta$-globally $(K'',G)$-Ramsey, we have that $U_2'$ contains a red copy of $K''$. We may assume that one of the edges of the copy of $G$ in $F_1$ is red, and we then obtain a red copy of $K$. Therefore, all of the edges from $V(F_1)$ to $U_2'$ must be blue. By a similar argument, we may assume that all of the edges from $V(F_1)$ to $U_i'$ are blue for all $2\leq i\leq k-1$.

 By decreasing $\eps$ if necessary, we may assume that $\alpha > \eps$, where $\alpha$ is defined as above. Then since $|U_i'|\geq \alpha|V_i|$ for $2\leq i\leq k-1$, Lemma \ref{lem:LargeNeighborhoods} implies that $(U_i',U_j')$ is an $\frac{\eps}{\alpha}$-regular pair with density $\delta'>0$  for all $2\leq i< j\leq k-1$. Now, by an inductive argument on $k$, we may find a copy $F_i$ of $F$ in $U_i'$ for $2\leq i\leq k-2$ and a constant $\alpha'$ depending only on $\eps$ such that there exists a vertex set $U_{k-1}''\subset V_{k-1}$ of size $\alpha'\eta n$ for which all edges with endpoints coming from distinct vertex sets among $V(F_1),\dots, V(F_{k-2})$ and $U_k''$ are blue. (Indeed, the base case $k=3$ was essentially proved above.) Furthermore, we may assume that each $F_i$ contains at least one blue edge, as otherwise it contains a red copy of $K$. However, by our assumptions on $G$, this implies that $\Gamma\cup G(n,p)$ contains a blue copy of $G$. Therefore, $\Gamma\cup G(n,p)$ is $(K,G)$-Ramsey with high probability.
\end{proof}

\section{Applications of Theorems \ref{thm:GeneralkPartition} and \ref{thm:main2}}\label{sec:applications}

In this section, we show that Theorem \ref{thm:GeneralkPartition} can be applied when $K$ is a complete graph and $G$ is a complete graph minus a perfect matching. Denoting by $K_s'$ the complete graph on $s$ vertices minus a maximum matching, we determine the threshold function $p(n;K_t,K_s',d)$ for the range of densities $\frac{k-2}{k-1}\leq d\leq \frac{k-1}{k}$ for each integer $k$ such that $2\leq k < \frac{s}{3}$. 

\begin{corollary}\label{cor:KsMinusMatching}
    Let $K_s'$ be the graph obtained by deleting a maximum matching from the complete graph $K_s$. For an integer $2\leq k < \frac{s}{3}$ and in the range of densities $\frac{k-2}{k-1}\leq d\leq \frac{k-1}{k}$, we have that for $t\geq s\geq 7$, 
    \[
    p(n;K_t,K_s',d) = n^{-1/m_2(K_t,K_{\lceil s/k \rceil}')}.
    \]
\end{corollary}
\begin{proof}
    It suffices to show that the hypothesis of Theorem \ref{thm:GeneralkPartition} is satisfied when we take $K=K_t$, $G = K_s'$ and $H=K_{\lceil s/k \rceil}'$. It is clear that in every partition of $V(K_s')$ into $k$ sets $V_1,\dots, V_k$, some part contains $K_{\lceil s/k\rceil}'$, and that taking $k$ copies of $K_{\lceil s/k\rceil}'$ and adding all possible edges between each pair of copies contains $K_s'$.

    Routine calculation shows that $m_2(K_t)\geq m_2(K_s'),m_2(K_{\lceil s/k \rceil}')$ and that $K_t$ is strictly 2-balanced with respect to $m_2(\cdot, K_s')$ and $m_2(\cdot, K_{\lceil s/k \rceil}')$.
    
    Similarly, $m_2(K_{t-1})\geq m_2(K_s')$, and $K_{t-1}$ is strictly 2-balanced with respect to $K_s'$. It is also routine to check that $m_2(K_t, K_{\lceil s/k \rceil}') \geq m_2(K_{t-1}, K_s')$. 

    Finally, by Theorem 1.5 in \cite{BHH}, the 0-statement for Conjecture \ref{conj:Gnpthreshold} holds for the pair $(K_t,K_{\lceil s/k \rceil}')$.

    Therefore, we may apply Theorem \ref{thm:GeneralkPartition}, and this completes the proof.
\end{proof}

\begin{remark}
    We require that $k< \frac{s}{3}$ in Theorem \ref{thm:GeneralkPartition} since if $k \geq \lceil\frac{s}{3}\rceil$. then $K_{\lceil\frac{s}{3}\rceil}'$ is a subgraph of the path on 3 vertices. Hence, $m_2(K_{\lceil\frac{s}{3}\rceil}') \leq 1$, so the statement of Conjecture \ref{conj:Gnpthreshold} does not apply to the pair $(K_t, K_{\lceil\frac{s}{3}\rceil}')$ when $k\geq \frac{s}{3}$.
\end{remark}

Additionally, in \cite{DT} Das and Treglown proved that for an integer $2\leq k < \frac{s}{2}$ and for the range of densities $\frac{k-2}{k-1}\leq d\leq \frac{k-1}{k}$, if $k=2$ or $s \equiv 1$ (mod $k$), then
\[
p(n;K_t,K_s,d) = n^{-1/m_2(K_t,K_{\lceil s/k \rceil})}.
\]
However, if $k\geq 3$ and $s \not \equiv 1$ (mod $k$), they only show that
\[
p(n;K_t,K_s,d) = n^{-(1-o(1))/m_2(K_t,K_{\lceil s/k \rceil})}.
\]
We are able to use Theorem \ref{thm:GeneralkPartition} to remove the $o(1)$ term in the expression, and thus obtain an exact threshold function when $k\geq 3$ and $s \not \equiv 1$ (mod $k$).

\begin{corollary}\label{cor:KtKs}
    Let $s,t,k$ be integers with $k\geq 2$ and $t\geq s\geq 2k+1$, and let $\frac{k-2}{k-1}\leq d\leq \frac{k-1}{k}$. Then
    \[
    p(n;K_t,K_s,d) = n^{-1/m_2(K_t,K_{\lceil s/k \rceil})}.
    \]    
\end{corollary}
\begin{proof}
    The proof is similar to the proof of Corollary \ref{cor:KsMinusMatching}, i.e., it is a straightforward application of Theorem \ref{thm:GeneralkPartition}.
\end{proof}

Das and Treglown also determined the threshold function for pairs of the form $(K_t,C_{2k+1})$ \cite{DT}. Specifically, they showed that for $0<d\leq 1/2$,
\[
p(n;K_t,C_{2k+1},d) = n^{-2/(t-1)}.
\]
Using Theorem \ref{thm:main2}, we generalize this result to pairs $(K_t,G)$ for $t\geq 5$ including, for example, the cases where $G$ is an odd cycle, an odd cycle with a chord, an even cycle with a chord connecting two vertices an even distance away along the cycle, or a star on at least 3 vertices with an additional vertex that is adjacent to each vertex of the star.

\begin{corollary}\label{cor:oddcycleGen}
    Given a tree $T$ with at least 4 vertices, let $G$ be a graph obtained by adding a vertex $v$ to $T$ and edges from $v$ to $V(T)$ such that there is no 2-coloring of $T$ where the vertices in $N_G(v)$ are all the same color, but there is some vertex $u \in N_G(v)$ such that there is a 2-coloring of $T$ where the vertices of $N_G(v)\setminus \{u\}$ are colored with the same color.
    
    Then for $t\geq 5$ and $0<d\leq 1/2$,
    \[
    p(n;K_t,G,d) = n^{-2/(t-1)}.
    \]

\end{corollary}
\begin{proof}
    This follows from applying Theorem \ref{thm:main2} when $K= K_t$ and $G$ is given as in the statement of this corollary. The graph $G'$ is obtained by deleting $v$, hence $G' = T$. We also have that $K' = K_{t-1}$ and $K'' = K_{t-2}$.

    Here, $d(K) = \frac{t-1}{2}$, and it is straightforward to verify that $\frac{t-1}{2} \geq m_2(K_{t-1},T) = \frac{t-1}{2}$, $m_2(K_{t-1})\geq m_2(T) =1$, and $K_{t-1}$ is strictly 2-balanced with respect to $m_2(\cdot,G')$.

    Note that $m_2(G) \leq 2$, and it is straightforward to verify that $m_2(K_{t-2})\geq m_2(G)\geq 1$ and $K_{t-2}$ is strictly 2-balanced with respect to $m_2(\cdot,G)$.

    Notice that $G$ is 3-chromatic, so we are applying Theorem \ref{thm:main2} for $k=3$. Take a 2-coloring of $T$ where the vertices $N_G(v)\setminus \{u\}$ all have the same color, say red. Let $V_1$ be the set of red vertices and $V_2$ be the set of blue vertices, and add all possible edges in between $V_1$ and $V_2$. Then add $v$ to $V_2$, add an edge between $v$ and $u$ ($u$ must be colored blue and hence is in $V_2$), then add all edges between $v$ and $V_1$. The resulting graph contains $G$ as a subgraph, so the last condition in Theorem \ref{thm:main2} is satisfied. This completes the proof.
\end{proof}

In the following corollary, we present another application of Theorem \ref{thm:main2} where $K$ has significantly fewer edges than a complete graph.

\begin{corollary}\label{cor:last}
    Let $K$ be the graph obtained by taking a star on $t\geq 4$ vertices and adding a new vertex that is adjacent to each vertex of the star. Let $G$ be as in Corollary \ref{cor:oddcycleGen}. Then for $0<d\leq 1/2$,
    \[
    p(n;K,G,d) = n^{-(t+1)/(2t-1)}.
    \]
\end{corollary}
\begin{proof}
    Let $T$ be a star on $t\geq 4$ vertices, and let $v$ be the vertex of $K$ that is adjacent to each vertex in $T$. We have that $d(K) = \frac{2t-1}{t+1}$ and $m_2(K) = 2$. Let $K' = T$, and let $G'$ be as in the proof of Corollary \ref{cor:oddcycleGen}. Then $d(K) \geq m_2(K',G') = 1$ and $m_2(K') = m_2(G') = 1$.

    Let $K''$ be obtained by deleting 2 vertices of $T$ from $K$. It is straightforward to check that $d(K) \geq m_2(K'',G)$, $m_2(K'') \geq m_2(G) \geq 1$ and $K''$ is strictly 2-balanced with respect to $m_2(\cdot, G)$.

    Finally, $G$ satisfies the last condition of Theorem \ref{thm:main2} as in the proof of Corollary \ref{cor:oddcycleGen}.

    This completes the proof since $n^{-1/d(K)} = n^{-(t+1)/(2t-1)}$.
\end{proof}

\section*{Acknowledgements}
Work on this project started during the Research Training Group (RTG) rotation at Iowa State University. Both authors were supported by NSF grant DMS-1839918. 
We would like to thank Ryan Martin for helpful discussions during early stages of this project. 

\small{

}

\end{document}